\documentclass{gLMA2eDoctored}

\theoremstyle{plain}
\newtheorem{theorem}{Theorem}[section]

\theoremstyle{remark}

\def\rmdigits{\mathcode`0="0030\mathcode`1="0031\mathcode`2="0032\mathcode`3="0033\mathcode`4="0034                 
              \mathcode`5="0035\mathcode`6="0036\mathcode`7="0037\mathcode`8="0038\mathcode`9="0039}	
\def\b#1{{\bf #1}}
\def\specradius#1{\rho\left({\uppercase{\bf #1}}\right)}
\def\specradiussub#1#2{\rho\left({\uppercase{\bf #1}}_{\rmdigits{\mit #2}}\right)}
\def\spectrum#1{\sigma\left({\uppercase{\bf #1}}\right)}

\def\matsub#1#2{{\uppercase{\bf #1}}_{\rmdigits{\mit\lowercase{#2}}}}
\def\vecsub#1#2{{\lowercase{\bf #1}}_{\rmdigits{\mit\lowercase{#2}}}}

\def\size#1#2{{\mit\rmdigits {#1}\times {#2}}}
\def\r#1{{\mit\rmdigits\Re^{#1}}}
\def\implies{~~\Longrightarrow~~}

\def\and{~~~~{\rm and}~~~~}
\def\norm#1{\left\|#1\right\|}
\def\onenorm#1{\left\|#1\right\|_1}
\def\twonorm#1{\left\|#1\right\|_2}
\font\bb=msbm10.pfb               

\def\r#1{\hbox{\bb R}^{#1}}

\def\eqalign#1{\null\,\vcenter{\openup1\jot \m@th
   \ialign{\strut\hfil$\displaystyle{##}$&$\displaystyle{{}##}$\hfil
      \crcr#1\crcr}}\,}
\def\m@th{\mathsurround=0pt}

\def\pmatrix#1{\left(\null\,\vcenter{\normalbaselines\m@th
    \ialign{$##$&&\quad$##$\crcr
      \mathstrut\crcr\noalign{\kern-\baselineskip}
      #1\crcr\mathstrut\crcr\noalign{\kern-\baselineskip}}}\,\right)} 

\def\rmatrix#1{\left(\null\,\vcenter{\normalbaselines\m@th
    \ialign{\hfill$##$\hfil&&\quad\hfill$##$\hfil\crcr
      \mathstrut\crcr\noalign{\kern-\baselineskip}
      #1\crcr\mathstrut\crcr\noalign{\kern-\baselineskip}}}\,\right)} 

\def\nakedrmatrix#1{\null\,\vcenter{\normalbaselines\m@th
    \ialign{\hfill$##$\hfil&&\quad\hfill$##$\hfil\crcr
      \mathstrut\crcr\noalign{\kern-\baselineskip}
      #1\crcr\mathstrut\crcr\noalign{\kern-\baselineskip}}}\,}        

\def\lmatrix#1{\left(\null\,\vcenter{\normalbaselines\m@th
    \ialign{\hfil$##$\hfill&&\quad\hfil$##$\hfill\crcr
      \mathstrut\crcr\noalign{\kern-\baselineskip}
      #1\crcr\mathstrut\crcr\noalign{\kern-\baselineskip}}}\,\right)} 

\def\nakedlmatrix#1{\null\,\vcenter{\normalbaselines\m@th
    \ialign{\hfil$##$\hfill&&\quad\hfil$##$\hfill\crcr
      \mathstrut\crcr\noalign{\kern-\baselineskip}
      #1\crcr\mathstrut\crcr\noalign{\kern-\baselineskip}}}\,}        

\def\sevenpmatrix#1{{\sevenpoint\pmatrix{#1}}}  

\def\phrase#1{\hbox{\quad {#1} \quad}}          
\def\phraseoverarrow#1{\setbox0=\hbox{\enspace #1\enspace}
                       \setbox1=\hbox to \wd0{\rightarrowfill}
                       \buildrel{\box0}\over{\box1}
                       }

\font\eightrm=cmr8
\font\eightmit=cmmi8
\font\eightsy=cmsy8

\font\eightbf=cmbx8
\font\eightit=cmti8
\font\eightsl=cmsl8
\font\eightcmbxti=cmbxti10 at 8pt

\font\sevenrm=cmr7
\font\sevenmit=cmmi7
\font\sevensy=cmsy7
\font\sevenex=cmex7
\font\sevenit=cmti7
\font\sevensl=cmsl8 at 7pt
\font\sevenbf=cmbx7
\font\sevencmbxti=cmbxti10 at 7pt

\font\sixrm=cmr6
\font\sixmit=cmmi6
\font\sixsy=cmsy6
\font\sixex=cmex10 at 6pt
\font\sixbf=cmbx6
\font\sixit=cmti10 at 6pt
\font\sixsl=cmsl10 at 6pt

\font\fiverm=cmr5
\font\fivemit=cmmi5
\font\fivesy=cmsy5
\font\fiveex=cmex10 at 5pt
\font\fivebf=cmbx5
\font\fiveit=cmti10 at 5pt
\font\fivesl=cmsl10 at 5pt

\font\fourrm=cmr5 at 4pt
\font\fourmit=cmmi5 at 4pt
\font\foursy=cmsy5 at 4pt
\font\fourex=cmex7 at 4pt
\font\fourit=cmti7 at 4pt
\font\foursl=cmsl8 at 4pt
\font\fourbf=cmbx5 at 4pt

\def\sevenpoint{%
    \textfont0=\sevenrm    \scriptfont0=\fiverm    \scriptscriptfont0=\fourrm%
    \textfont1=\sevenmit   \scriptfont1=\fivemit   \scriptscriptfont1=\fourmit%
    \textfont2=\sevensy    \scriptfont2=\fivesy    \scriptscriptfont2=\foursy%
    \textfont3=\sevenex   \scriptfont3=\fiveex    \scriptscriptfont3=\fourex%
    \textfont4=\sevenit    \scriptfont4=\fiveit    \scriptscriptfont4=\fourit%
    \textfont5=\sevensl    \scriptfont5=\fivesl    \scriptscriptfont5=\foursl%
    \textfont6=\sevenbf    \scriptfont6=\fivebf    \scriptscriptfont6=\fourbf%
  \baselineskip=8pt
  \normalbaselineskip=8pt
  \lineskiplimit=1pt
  \lineskip=2pt minus .5pt
  \parindent=0pt%
  \sevenrm%
  \let\it=\sevenit%
  \let\bb=\sevenbb%
  \let\timesbi=\sevencmbxti%
  \def\bf{\fam=6\sevenbf}%
  }
\def\eightpoint{%
    \textfont0=\eightrm    \scriptfont0=\sixrm    \scriptscriptfont0=\fiverm%
    \textfont1=\eightmit   \scriptfont1=\sixmit   \scriptscriptfont1=\fivemit%
    \textfont2=\eightsy    \scriptfont2=\sixsy    \scriptscriptfont2=\fivesy%
    \textfont3=\tenex      \scriptfont3=\sixex    \scriptscriptfont3=\fiveex%
    \textfont4=\eightit    \scriptfont4=\sixit    \scriptscriptfont4=\fiveit%
    \textfont5=\eightsl    \scriptfont5=\sixsl    \scriptscriptfont5=\fivesl%
    \textfont6=\eightbf    \scriptfont6=\sixbf    \scriptscriptfont6=\fivebf%
  \baselineskip=9pt
  \normalbaselineskip=9pt
  \lineskiplimit=1pt
  \lineskip=2pt minus .5pt
  \parindent=0pt%
  \eightrm%
  \let\it=\eightit%
  \let\bb=\eightbb%
  \let\timesbi=\eightcmbxti%
  \def\bf{\fam=6\eightbf}
}
\begin{document}

\title{{\itshape Continuity of the Perron Root}}

\author{Carl D. Meyer$^{\ast}$\thanks{$^\ast$Corresponding author. Email: meyer@ncsu.edu\vspace{6pt}}
\\\vspace{6pt}  $^{\ast}${\em{Department of Mathematics, North Carolina State University, Raleigh, NC 27695}}
}

\maketitle

\begin{abstract}
That the Perron root of a square nonnegative matrix $\b A$ varies continuously with the entries in $\b A$ is a corollary
of theorems regarding continuity of eigenvalues or roots of polynomial equations, the proofs of which necessarily
involve complex numbers. But since continuity of the Perron root is a question that is entirely in the field of real numbers, it seems
reasonable that there should exist a development involving only real analysis. 
This article presents a simple and completely self-contained development that depends only on real numbers and first principles.


\begin{classcode}1502; 15A18; 15B48 \end{classcode}

\end{abstract}

\section{Introduction}
The spectral radius \hbox{$r=\specradius a$} of a square matrix with nonnegative entries is called the
{\it Perron root} of $\b A$ because the celebrated Perron--Frobenius theory
(summarized below in \S\ref{PerronFrobeniusBasics}) guarantees that $r$ is an eigenvalue for $\b A.$
If $\{\matsub ak\}_{k=1}^\infty$ is a sequence of $\size nn$ nonnegative matrices with respective Perron roots $r_k,$ and if
$\lim_{k\to\infty}\matsub ak=\b A,$  then it seems rather intuitive that $\lim_{k\to\infty}r_k=r,$ for otherwise
something would be dreadfully wrong. But this is not a proof. In fact, a simple self-contained proof 
depending only on first principles that are strictly in the realm of real numbers seems to have been elusive.

The standard treatment is usually to pawn off the result as a corollary to theorems regarding the continuity of eigenvalues
for general matrices.
For example, citing the continuity of roots of polynomial equations is an easy dodge, but it fails to satisfy because
it buries the issue under complex analysis involving Rouch\'e's theorem which itself requires the argument principle. 
And then there is Kato's development \cite{Kato} of eigenvalue continuity built around resolvent integrals, which also requires some heavy
lifting with complex analysis.  The approach to continuity in \cite{HornAndJohnson} that utilizes Schur's decomposition in terms of unitary
matrices is concise and can be cited or adapted, but it too must necessarily venture outside the realm of real numbers because there is no real version of Schur's theorem that does the job.

While all reference to complex numbers cannot be completely expunged (e.g., nonnegative matrices can certainly have complex
eigenvalues, and the definition of spectral radius given in PF1 below is dependent on them),
it is nevertheless true that the
continuity of the Perron root is an issue that is entirely in the realm of real numbers, so it seems only
reasonable that there should be a simple argument involving only real analysis. The purpose
of this article is to present a simple and completely self-contained development that is strictly in the real domain
and depends only on rudimentary principles from real analysis together with
basic Perron--Frobenius facts as summarized below.

\section{Perron--Frobenius Basics}\label{PerronFrobeniusBasics}
The only Perron--Frobenius facts required to establish the continuity of the Perron root are given here.
Details and the complete theory can be found in \cite[Chapter 8]{MeyerText}.
If $\b A_{\size nn}\geq\b 0$ (entrywise) whose spectrum is $\spectrum a,$ then:

\medskip
\begin{enumerate}
\item[PF1.] The spectral radius $r=\specradius a=\displaystyle\max\big\{|\lambda|\,\big|\,\lambda\in\spectrum a\big\}$
is an eigenvalue \hbox{for $\b A.$} 

\medskip
\item[PF2.] There is an associated eigenvector  $\b x\neq\b 0$ such that $\b A\b x=r\b x,$
where $\b x\geq\b 0.$ Such vectors can always be normalized so that $\onenorm{\b x}=1,$ and when this
is done, the resulting eigenvector is referred to as a {\it Perron vector} for $\b A.$

\medskip
\item[PF3.] If $\b A$ is irreducible (i.e., no permutation similarity transformation of $\b A$
can produce a block triangular form with square diagonal blocks), then $r>0$ and $\b x>0$ when $n\geq 2.$

\medskip
\item[PF4.] If $\b 0\leq\b A\leq\b B$ (entrywise), then $\specradius a\leq\specradius b.$ In particular, if $\b A$ is a square
submatrix of $\b B,$ then $\specradius a\leq\specradius b.$

\end{enumerate}

\section{The Development}
Throughout, let $\{\matsub ak\}_{k=1}^\infty$ be a sequence of $\size nn$ nonnegative matrices
with respective Perron roots $r_k,$ and assume that $\lim_{k\to\infty}\matsub ak=\b A.$ 
The aim is to prove that $\lim_{k\to\infty} r_k=r,$ where $r$ is the Perron root of $\b A.$  

Since each $\matsub ak\geq\b 0,$ 
it is apparent that $\b A\geq\b 0,$  so the argument can be
divided into two cases (or theorems) in which (1) $\b A$ is nonnegative and irreducible; and
(2) $\b A$ is nonnegative and reducible.

\subsection*{The Irreducible Case}
When $\b A$ is irreducible, the proof is essentially a ``one-liner.''

\begin{theorem}\label{IrreducibleCase}
If $\b A$ is irreducible, then $r_k\to r.$
\end{theorem}

\begin{proof}
Let $\matsub ek=\matsub ak-\b A,$ and let $\b p_k$ and $\b q^T$ be respective right- and left-hand
Perron vectors for $\b A_k$ and $\b A$ with $\onenorm{\b p_k}=1=\onenorm{\b q}.$ If $q_\star=\min q_i,$
and if $\b e$ is a vector of ones, then
\hbox{$\b q\geq q_\star\b e,$}  and
     $$
     \b q^T\b p_k\geq q_\star\b e^T\b p_k=q_\star>0
     \quad\hbox{for all $k.$}
     $$
Using this with the Cauchy--Schwarz inequality and  $\twonorm{\b x}\leq\onenorm{\b x}$ for all $\b x\in\r n$ yields
     $$
     \eqalign{
       |(r_k-r)\b q^T\vecsub pk|
       &=|\b q^T(r_k\vecsub pk)-(r\b q^T)\vecsub pk|
       =|\b q^T(\matsub ak \vecsub pk)-(\b q^T\b A)\vecsub pk|\cr
       &=|\b q^T(\matsub ak-\b A)\vecsub pk|
       =|\b q^T\matsub ek\vecsub pk|
       \leq\twonorm{\matsub ek} \cr
       \implies
       &|r_k-r|\leq{\twonorm{\matsub ek}\over \b q^T\vecsub pk}
       \leq{\twonorm{\matsub ek}\over q_\star}
       \to 0
       \implies r_k\to r.
       \cr}
     $$
\vskip-30pt\end{proof}

\subsection*{The Reducible Case}

When $\b A$ is reducible, the proof requires a few more lines than the irreducible case.

\begin{theorem}[The Reducible Case]\label{ReducibleCase}
If $\b A$ is reducible, then $r_k\to r.$
\end{theorem}

\begin{proof}
If $r=0,$ then $\b A$ is nilpotent, say $\b A^p=\b 0,$ so
     $$
       \big[r_k\big]^p
       =\big[\rho(\matsub ak)\big]^p
       =\rho(\matsub ak^p)
       \leq\norm{\matsub ak^p}
       \to\norm{\b A^p}
       = 0
       \implies r_k\to 0=r.
     $$
Now assume that $r>0.$
The foundation for the remaining part of the proof rests on the following realization.
\begin{eqnarray}\label{TheLemma}
\left\{\nakedlmatrix
{
\hbox{\em Every subsequence $\{r_{k_i}\}$ of $\{r_k\}$ has a sub-subsequence $\{r_{k_{i_j}}\}$ such that} \cr
\noalign{\vskip 4pt}
\hfill\hbox{$r_{k_{i_j}}\to r.$}
\cr}
\right\}   
\end{eqnarray}
To establish this, adopt the notation
$\b X\sim\b Y$ to mean that $\b Y=\b P^T\b X\b P$ for some permutation matrix $\b P$ so that
\hbox{$\b A\sim\sevenpmatrix{\b U& \b V\cr \b 0&\b W\cr},$}  where $\b U$ and $\b W$
are square. If either $\b U$ or $\b W$ is reducible, then they in turn can be reduced in the same fashion.
Reduction of diagonal blocks can continue until at some point
     \begin{eqnarray}
       \b A\sim
       \sevenpmatrix{
       \bullet&\cdots&\bullet&\cdots&\bullet\cr
       \vdots&\ddots&\vdots&&\vdots\cr
       \b 0&\cdots&\b B&\cdots&\bullet\cr
       \vdots&&\vdots&\ddots&\vdots\cr
       \b 0&\cdots&\b 0&\cdots&\bullet\cr
       }
       \label{equationReductionOfReducibleMatrix}
     \end{eqnarray}
is block triangular with square diagonal blocks, one of which---call it $\b B$---is necessarily
irreducible and has  $\specradius b=\specradius a=r>0.$
Apply the same symmetric permutation that produced (\ref{equationReductionOfReducibleMatrix}) 
to each $\matsub ak$ so that
     $$
       \matsub ak\sim
       \sevenpmatrix{
       \bullet&\cdots&\bullet&\cdots&\bullet\cr
       \vdots&\ddots&\vdots&&\vdots\cr
       \bullet&\cdots&\kern2pt\b B_k&\cdots&\bullet\cr
       \vdots&&\vdots&\ddots&\vdots\cr
       \bullet&\cdots&\bullet&\cdots&\bullet\cr
       },
     $$
where $\matsub bk$ and $\b B$ have the same size and occupy the same positions. 
It follows from (PF4) that if $b_k=\specradiussub bk,$ then
$b_k\leq r_k$ for each $k.$
And since $\matsub ak\to\b A$ implies $\matsub bk\to\b B,$ Case (1) (the irreducible case) ensures
$\specradiussub bk\to\specradius b$ so that $b_k\to r.$
In particular, if $\{r_{k_i}\}$ is any subsequence of $\{r_k\},$ then
\begin{eqnarray}\label{equationBkiGoesToR}
b_{k_i}\leq r_{k_i}\ \hbox{for each $k_i,$}
\and
b_{k_i}\to r.
\end{eqnarray}
Every subsequence $\{r_{k_i}\}$ is bounded because
$\specradius\star\leq\norm{\star}$ for any matrix norm, and this implies that
     $
     0\leq r_{k_i}\leq\norm{\matsub a{k_i}}=\norm{\b A+\matsub e{k_i}}
     \to\norm{\b A}.
     $
Hence every subsequence $\{r_{k_i}\}$ has a convergent sub-subsequence $r_{k_{i_j}}\to r^\star.$ This together with
(\ref{equationBkiGoesToR}) yields
     \begin{eqnarray}
     b_{k_{i_j}}\leq r_{k_{i_j}}
     \phrase{so that}
     r\leq r^\star.
     \label{equationRleqRstar}
     \end{eqnarray}
To see that $r^\star=r,$ note that
the sequence of Perron vectors $\{\vecsub v{k_i}\}$ for $\matsub a{k_i}$ is bounded because
each has norm one, so $\{\vecsub v{k_i}\}$ has a convergent subsequence \hbox{$\vecsub v{k_{i_j}}\to\b v^\star\neq\b 0.$} 
Use this together with $r_{k_{i_j}}\to r^\star$ to conclude that
     $$
     \eqalign{
       \b A\b v^\star
       &= \lim\matsub A{k_{i_j}}\kern-2pt\lim\vecsub v{k_{i_j}}
       = \lim[\matsub A{k_{i_j}}\kern-3pt\vecsub v{k_{i_j}}]
       = \lim [r_{k_{i_j}}\kern-3pt\vecsub v{k_{i_j}}]
       = \lim r_{k_{i_j}}\kern-2pt\lim\vecsub v{k_{i_j}}
       = r^\star\b v^\star \cr
       &\implies
       \hbox{$r^\star$ is an eigenvalue for $\b A$} 
       \implies
       r^\star\leq r.
       \cr
       }
     $$
This together with (\ref{equationRleqRstar}) 
ensures that $r^\star=r,$ and thus (\ref{TheLemma}) is established.
To prove that \hbox{$r_k\to r,$}  suppose to the contrary that
$r_k\not\to r$ so that there is a subsequence $\{r_{k_s}\}$ and a number $\epsilon>0$ such that
$|r_{k_s}-r|>\epsilon$ for all $s=1,2,3,\ldots.$ However, (\ref{TheLemma}) guarantees that $\{r_{k_s}\}$ has a subsequence
$\{r_{k_{s_j}}\}$ such that $r_{k_{s_j}}\to r,$ which is a contradiction, and thus $r_k\to r.$
\end{proof}

\section{A Temptation to Avoid}
It is tempting to establish the continuity of the Perron root (or the spectral radius in general)
by using the characterization 
     \begin{eqnarray}\label{AltDefOfSpecRadius}
       \lim_{m\to\infty}\norm{\b X^m}^{1/m}=\specradius x
     \end{eqnarray} 
to simply conclude that
     $$
       \lim_{k\to\infty} r_k
       =\lim_{k\to\infty}\lim_{m\to\infty}\norm{\matsub ak^m}^{1/m}
       =\lim_{m\to\infty}\lim_{k\to\infty}\norm{\matsub ak^m}^{1/m}
       =\lim_{m\to\infty}\norm{\b A^m}^{1/m}
       =r.
     $$
It would be acceptable to interchange the limits on $k$ and $m$ if the convergence in (\ref{AltDefOfSpecRadius}) was uniform
on the set $\cal N=\{\b X\in\r{\size nn}\,|\,\b X\geq\b 0\},$
but alas, it is not. To see this, observe that if $f_m(\b X)=\norm{\b X^m}^{1/m}$ and $f(\b X)=\specradius x,$ then 
$f_m(\alpha\b X)=\alpha f_m(\b X)$ and $f(\alpha\b X)=\alpha f(\b X)$
for all $\alpha\geq 0$ and for all $\b X\in\cal N.$
The convergence of $f_m$ to $f$ cannot be uniform because otherwise, for each $\epsilon>0,$ there would exist an integer $M$ such that
$m\geq M$ implies
     $$
       \left|f_m(\b X)-f(\b X)\right|<\epsilon
       \quad\hbox{for all $\b X\in\cal N.$} 
     $$
In particular, $m\geq M$ implies that
     $$
     \left|f_m(\alpha\b X)-f(\alpha\b X)\right|<\epsilon
     \quad\hbox{for all $\alpha\geq 0$ and  $\b X\in\cal N,$} 
     $$
or equivalently,
     $$
       \alpha\,\left|f_m(\b X)-f(\b X)\right|
       <\epsilon
       \quad\hbox{for all $\alpha\geq 0$ and  $\b X\in\cal N,$}
     $$
which is impossible.

\section{Concluding Remarks}
The developments given in this article provide a simple proof
that is entirely contained in the real domain for establishing the continuity of the Perron root of a nonnegative matrix.
However, they do not apply for proving the continuity of the spectral radius in general---for this complex analysis cannot
be avoided.  Even if the limit of a sequence of general matrices is nonnegative,
the spectral radius can correspond to a complex eigenvalue so that the techniques of Theorems
\ref{IrreducibleCase} and \ref{ReducibleCase} do not apply---e.g., consider
     \begin{eqnarray}\label{Example}
       \matsub ak=\pmatrix{\kern 12pt0&1&0\cr -1/k&0&1\cr \kern 12pt 1&0&0\cr}
       \to\pmatrix{0&1&0\cr 0&0&1\cr 1&0&0\cr}.
     \end{eqnarray} 

Even though the continuity of the Perron root is not new, there is nevertheless current interest in extensions such as those given
in \cite{LemmensAndNussbaum}.

\section*{Acknowledgements}
The author wishes to thank the referee for providing suggestions and corrections
that enhanced the exposition. The referee is also responsible for
example (\ref{Example}), and for pointing out
the work in \cite{LemmensAndNussbaum}. In addition, thanks are extended to
Stephen Campbell for suggesting the simple explanation of why
the convergence of (\ref{AltDefOfSpecRadius}) is not uniform.

\end{document}